\documentclass[10pt,twoside]{homg3} 
\usepackage{amssymb,graphicx}

\renewcommand{\R}{\mathbb{R}}

\renewcommand{\B}{{\cal B}}

\usepackage[spanish]{babel}
\usepackage{color,fancybox,float,graphicx,latexsym,srcltx,times,tikz,url}
\newcommand{\id}{\rm{id}}
\newcommand{\tr}{\rm{tr}}
\newcommand{\diag}{\rm{diag}}
\newcommand{\suctres}[1]{{#1_{1},#1_{2},#1_{3}}}
\newcommand{\col}[3]
{\left( \begin{array}{c} {#1}\\ {#2}\\ {#3}\\
\end{array}\right)}

\usepackage{epsfig,float}
\DeclareGraphicsExtensions{.bmp,.png,.pdf,.jpg,.jpeg}

\def\TituloCorto{Matrices de rotaciones y simetr\'{\i}as}
\def\AutorCorto{M.J. de la Puente}

\begin{document}

\tittle{Matrices de rotaciones, simetr\'{\i}as y roto--simetr\'{\i}as}

\def\authors{\aaa{Mar\'{\i}a Jes\'{u}s  D{\surname E LA PUENTE}
}}

\def\direc{\address{
Departamento de \'{A}lgebra\\
Facultad de Matem\'{a}ticas\\
Universidad Complutense\\
28040 Madrid, Spain \\
mpuente@ucm.es}}

\maketitle

\begin{center}
{\em Dedication}
\end{center}

\vspace{0.1cm}

\begin{abstract}
\noindent In this note we find the  orthogonal matrices $R,S\in M_3(\mathbb{R})$ corresponding to the clockwise rotation $r$  in $\mathbb{R}^3$ around the axis generated by a unit vector $u=(a,b,c)^t$ through an angle $\alpha\in [0,2\pi)$, and  to the symmetry $s$ in $\mathbb{R}^3$ on the plane perpendicular to  $u$. Matrix $S$ depends on  $a,b,c$  and matrix $R$  depends  on  $a,b,c, \cos \alpha$ and $\sin \alpha$. We show $SR=RS$. The matrix $R$ is due to Alperin.
\end{abstract}

\par
\bigskip



En los libros de \'{A}lgebra Lineal   al uso a nivel universitario, se encuentran diversos problemas en el espacio eucl\'{\i}deo $\R^3$ del siguiente tipo:
 \begin{enumerate}


 \item dados la recta  $E$ y el \'{a}ngulo $\alpha\in [0,2\pi)$,    hallar la matriz $R$ de la rotaci\'{o}n (o giro) en $\R^3$ alrededor del eje $E$, de amplitud $\alpha$ y sentido positivo,

 \item dado el plano  $H$,  hallar la matriz $S$ de la simetr\'{\i}a (o reflexi\'{o}n) en $\R^3$ sobre $H$,

\end{enumerate}
donde $E$, $\alpha$ y $H$ son  datos concretos.  Llamaremos a estos \emph{problemas directos}. Bajo hip\'{o}tesis  adecuadas, las matrices $R,S$ resultan ser ortogonales, i.e., $R^{-1}=R^t$ y $S^{-1}=S^t$.
Asimismo encontramos los \emph{problemas inversos}: dada la matriz ortogonal $M\in M_3(\R)$, averiguar si
$M$ representa una rotaci\'{o}n alrededor de un eje, o una simetr\'{\i}a, o la composici\'{o}n de las anteriores, indicando en cada caso los
elementos geom\'{e}tricos asociados (eje  $E$, amplitud $\alpha\in [0,2\pi)$  de la rotaci\'{o}n, plano   $H$ de
la simetr\'{\i}a, etc.) \footnote{Podemos hablar tambi\'{e}n de rotaci\'{o}n de amplitud $\alpha\in [0,\pi]$ con sentido positivo o negativo. La equivalencia es clara: la rotaci\'{o}n de amplitud $\alpha$ y sentido negativo coincide con  la rotaci\'{o}n de amplitud $2\pi-\alpha$ y sentido positivo.}

\medskip
En esta nota abordamos las preguntas anteriores en  general, lo que sirve  para responder    tanto
problemas directos como inversos. Obtendremos $S\in M_3(F_1)$ y  $R\in M_3(F_2)$, para cierta extensi\'{o}n de cuerpos $$\Q\subseteq F\subseteq F_1\subseteq F_2\subseteq \R,$$ donde $F$  es el cuerpo base del problema.

\medskip
Trabajaremos en un espacio vectorial real $V$ de dimensi\'{o}n 3, dotado de un producto escalar $\langle, \rangle$
 y usaremos siempre bases ortonormales. Las coordenadas de los vectores de $V$ (respecto de cualquier base)
 se escribir\'{a}n en columna. Para fijar ideas, el lector puede tomar $V=\R^3$ con  el producto escalar
 usual $\langle, \rangle$,  el producto vectorial usual $\wedge$ y la base can\'{o}nica $\{\suctres e\}$.
Recordemos la conocida igualdad
\begin{equation}\label{eqn:igualdad}
u\wedge(v\wedge w)=v\langle u,w\rangle -w\langle u,v\rangle,
\end{equation}
cuya demostraci\'{o}n (usando coordenadas) es un sencillo ejercicio.

\medskip
Fijemos una base ortonormal $\B$ de $V$.
Sean $(a,b,c)^t$  las coordenadas, respecto de $\B$, de un vector unitario $u$ que genere $E$ (i.e., $\langle u,u \rangle = a^2+b^2+c^2=1$). Denotemos por
$p_E$ la proyecci\'{o}n ortogonal de $\R^3$ sobre $E$.
 Denotemos por $r_{E,\alpha}$  y $s_{H}$ la rotaci\'{o}n y   simetr\'{\i}a  descritas m\'{a}s arriba,
donde el plano $H$, de ecuaci\'{o}n $ax+by+cz=0$, es precisamente $E^{\perp}$.

\begin{lemma}
La matriz de $p_E$ respecto de $\B$
es
\begin{equation}\label{eqn:pro}
A=\left(\begin{array}{ccc}
        a^2&ab&ac\\
        ab&b^2&bc\\
        ac&bc&c^2
    \end{array}\right)\in M_3(\Q(a,b,c)).
\end{equation}
\end{lemma}
\begin{proof}
Si $v\in V$ es un vector arbitrario, sabemos que $p_E(v)=u\langle u,v\rangle$, ya que $u$ es unitario. Por tanto,
si las coordenadas de $v$ respecto de $\B$ son $(x,y,z)^t$,  tenemos
$$p_E(v)=\col a b c  (a,b,c) \col x y z=
      A \col x y z.$$
\end{proof}

\begin{theorem}
La matriz de $r_{E,\alpha}$  es
\begin{equation}\label{eqn:rot}
R=I+(\sen \alpha)B+(\cos \alpha-1)(I-A)\in M_3(\Q(a,b,c,\cos\alpha,\sen \alpha)),
\end{equation}
y la matriz de $s_{E^{\perp}}$  es
\begin{equation}\label{eqn:sim}
S=I-2A\in M_3(\Q(a,b,c)),
\end{equation}
donde   $I$ denota la matriz identidad de orden 3 y
$$B=
    \left(\begin{array}{rrr}
        0&-c&b\\
        c&0&-a\\
        -b&a&0
    \end{array}\right).$$
\end{theorem}

\begin{proof}
La demostraci\'{o}n de (\ref{eqn:rot}) requiere varios pasos.
Para empezar, observemos que la aplicaci\'{o}n $g_u:V \to V$ tal que $v\mapsto u\wedge v$ es lineal. Adem\'{a}s,
la matriz de $g_u$ respecto de $\B$ es $B$.  Comprobamos que  $-B^2=I-A$, lo que significa que
\begin{equation}\label{eqn:menos_cuadrado}
-g_u^2={\id}-p_E=p_{E^{\perp}}
\end{equation}
ya que  las proyecciones sobre $E$ y sobre $E^{\perp}$ son complementarias.
A continuaci\'{o}n veamos que
\begin{equation}\label{eqn:rot_apl}
r_{E,\alpha}={\id} +(\sen \alpha) g_u+(\cos \alpha-1) p_{E^\perp},
\end{equation}
de donde se deducir\'{a} (\ref{eqn:rot}), gracias a (\ref{eqn:menos_cuadrado}) y al lema. Escribamos $f={\id} +(\sen \alpha)g_u+(\cos \alpha-1)p_{E^\perp}$ y demostremos
que $f$ y $r_{E,\alpha}$ act\'{u}an igual sobre los elementos de cierta base $\B'$ de $V$. Tomemos cualquier
vector  unitario $v$  perpendicular a $u$ y sea $\B'=\{u,v, u\wedge v\}$. Unos c\'{a}lculos sencillos (usando $u\wedge(u\wedge v)= u\langle u,v  \rangle  -v\langle u,u\rangle =-v$ a partir de (\ref{eqn:igualdad}))
nos muestran
que $f(u)=u$, $f(v)=(\cos \alpha)v+(\sen \alpha)( u\wedge v)$ y
$f(u\wedge v)=-(\sen \alpha)v+(\cos \alpha)( u\wedge v)$, de donde se sigue la igualdad (\ref{eqn:rot_apl}) y, con ella (\ref{eqn:rot}).

\medskip
Ahora tomemos  el vector $v=h^{-1}(-b,a,0)^t$ y consideramos la base
$\B'=\{u,v, u\wedge v\}$,
donde $h=\sqrt {a^2 +b^2}$. La matriz de $s_{E^\perp}$ respecto de $\B'$ es $D=\diag(-1,1,1)$, ya que $v$ y
$u\wedge v$ son  perpendiculares a $u$. Por tanto, la matriz de $s_{E^\perp}$ respecto de $\B'$ es
$S=PDP^{-1}$, donde
$$P=\left(\begin{array}{ccc}
a&-b/h&-ac/h\\
b&a/h&-bc/h\\
c&0&h
\end{array}\right)$$
 es matriz ortogonal (i.e., $P^{-1}=P^t$). Un c\'{a}lculo sencillo proporciona $S=PDP^t=I-2A$, que es la
 expresi\'{o}n (\ref{eqn:sim}).
\end{proof}

\begin{corollary}
Llamemos \emph{roto--simetr\'{\i}a}
a la composici\'{o}n $s_{E^\perp}\circ r_{E,\alpha}=r_{E,\alpha}\circ s_{E^\perp}$. Su matriz
es
\begin{equation}\label{eqn:rs}
SR=RS=S+(\sen \alpha)B+(\cos \alpha-1)(I-A)\in M_s(\Q(a,b,c,\cos \alpha,\sen \alpha)).
\end{equation}
\end{corollary}

\begin{proof}
Basta ver que   rotaci\'{o}n y simetr\'{\i}a  conmutan, y esto es cierto ya que  $SB=B=BS$ y
$S(I-A)=I-A=(I-A)S$, igualdades de comprobaci\'{o}n inmediata.
\end{proof}

\medskip
 Observaciones:
 \begin{enumerate}
 \item  \label{item:uno} En (\ref{eqn:sim}) tenemos $S=I-2A$, de donde se deduce la conocida relaci\'{o}n (ver figura \ref{fig})
    \begin{equation}\label{eqn:sim_proy}
     s_{E^{\perp}}={\id} - 2p_{E}.
     \end{equation}
 \item   El rango de la matriz $A$ es 1 y $A$ no es ortogonal. Se verifica $A^2=A=A^t$ e $(I-A)^2=I-A$. La
 matriz $B$ es antisim\'{e}trica y $-B^2=I-A$.  De aqu\'{\i} se sigue que las matrices  $R$ y $S$
 son ortogonales, i.e., $RR^t=I=SS^t=S^2$. \label{item:dos}
 \item Los determinantes y las trazas  valen ${\det} R=1$, ${\det} S={\det} (SR)=-1$, ${\tr} R=1+2\cos \alpha$, ${\tr} S=1$ y ${\tr} (SR)=-1+2\cos \alpha$. Determinante y traza son valores  invariantes de una isometr\'{\i}a de $\R^3$, pero NO la caracterizan en general (salvo que la traza valga 1). En efecto, el determinante nos dice si la isometr\'{\i}a  conserva o invierte la orientaci\'{o}n y la traza proporciona el valor $\cos \alpha$.  Queda, pues, por determinar el signo de $\sen \alpha=\pm \sqrt{\ 1-\cos ^2 \alpha}$. \label{item:tres}
 \end{enumerate}

\begin{figure}[h]
\centering
\includegraphics[width=8cm]{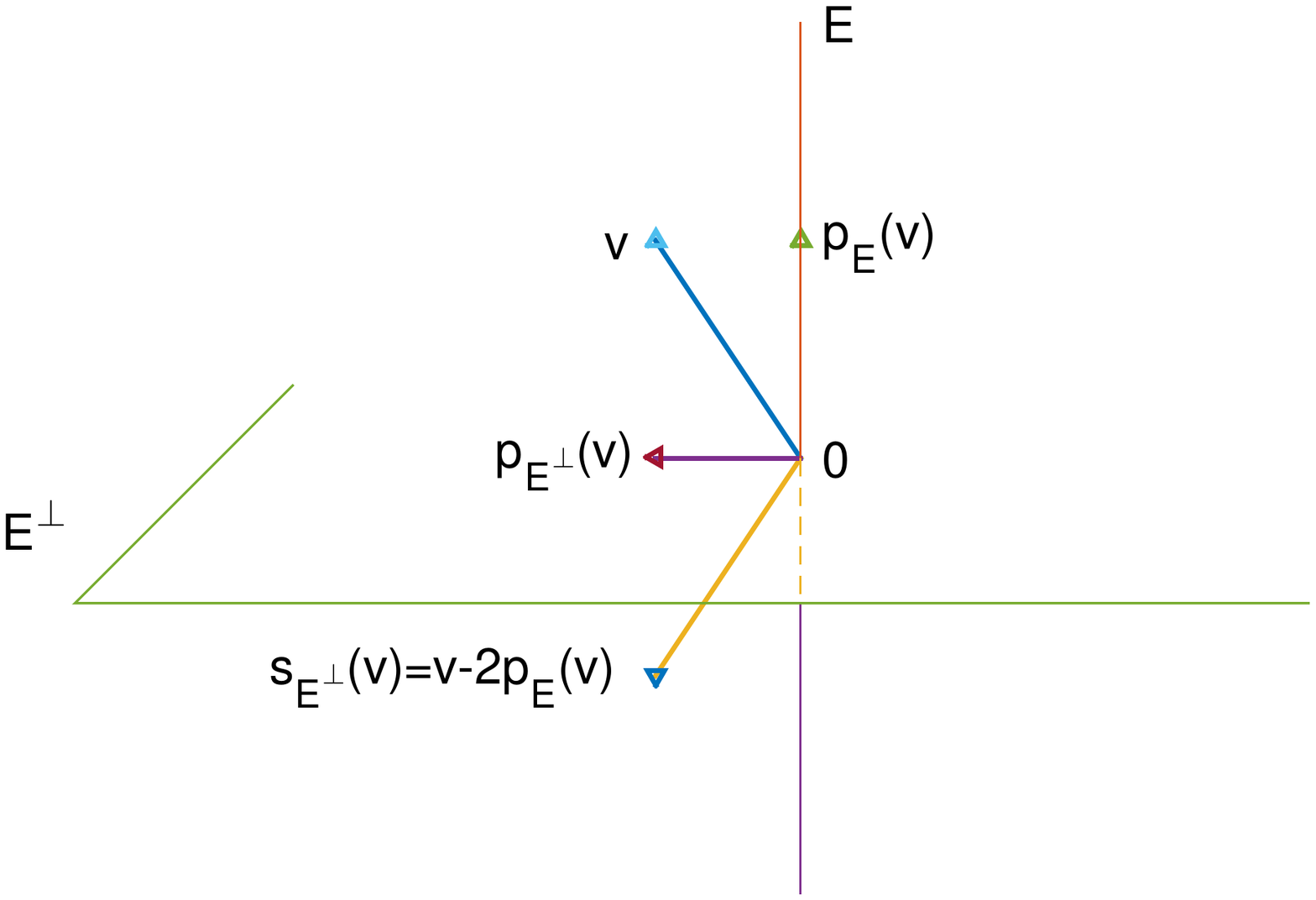}\\
\caption{Simetr\'{\i}a respecto del plano perpendicular a la recta  $E$.}
\label{fig}
\end{figure}

Ejemplo:
 Una matriz ortogonal sencilla (pero no trivial) es
$$M=\frac{1}{pq}\left(\begin{array}{rrr}
p&q&1\\
p&-q&1\\
p&0&-2
\end{array}\right)\in M_3(F),$$
donde $p=\sqrt{2}$ y $q=\sqrt{3}$ y $F=\Q(p,q)$ es el cuerpo base.
Como $MM^t=I$ y  $\det M=1$, sabemos que  $M$ representa una rotaci\'{o}n alrededor de un eje $E$. Vamos a hallar un
vector unitario $u=(a,b,c)^t$ que genere $E$ y la  amplitud de giro
$\alpha\in [0,2\pi)$ en sentido positivo. Ciertos c\'{a}lculos proporcionan
$$a=\frac{p+q}{n},\quad b=\frac{2-p-q+pq}{n}, \quad c=\frac{1}{n},$$
con $n=\sqrt{21-10p-8q+8pq}$.
Como la traza es invariante sabemos, por la observaci\'{o}n \ref{item:tres}, que
$$1+2\cos \alpha={\tr} M=\frac{-2+p-q}{pq}=-\frac{p}{2}+\frac{q}{3}-\frac{pq}{3},$$ de donde
$$\cos \alpha=-\frac{1}{2}-\frac{p}{4}+\frac{q}{6}-\frac{pq}{6},\quad
\sen \alpha=\pm \sqrt{1-\cos^2 \alpha}=\pm\frac{pqr}{12},$$ con $r=\sqrt{9-2p-2pq}$.
Sustituimos estos valores en (\ref{eqn:rot}) y  obtenemos la igualdad $M=R$ cuando el signo del seno es NEGATIVO, i.e.,  $\sen \alpha=-pqr/12$. Conocidos $\sen \alpha$ y $\cos \alpha$, deducimos que $\alpha\in (\pi,2\pi/3)$; concretamente $\alpha\simeq 193 ^\circ 20'$.

\medskip

Los libros de texto suelen  usar otro procedimiento para calcular $\alpha$. Hallan la
amplitud
$$\alpha=\arccos \left(-\frac{1}{2}-\frac{p}{4}+\frac{q}{6}-\frac{pq}{6}\right),$$ que tiene DOS soluciones,
$\alpha_1, \alpha_2\in [0,2\pi)$,  con $\alpha_1\le\alpha_2=2\pi-\alpha_1$. Luego determinan, mediante alg\'{u}n
argumento geom\'{e}trico, si $\alpha=\alpha_1$ \'{o}  $\alpha=\alpha_2$.
En cambio, nuestro razonamiento se ha basado en las f\'{o}rmulas (\ref{eqn:rot}), (\ref{eqn:sim}) y (\ref{eqn:rs}), trabajando   en la extensi\'{o}n algebraica $\Q(p,q,n,r)$ del cuerpo  base $\Q(p,q)$.
En nuestro caso tenemos $\cos \alpha\simeq -0.973126$, $\sen \alpha\simeq -0.230270$ y
$\alpha=\alpha_2\simeq 193 ^\circ 20'$.

\medskip
Agradecimientos: He tomado la matriz $R$ (as\'{\i} como su obtenci\'{o}n) de la parte debida a Roger C. Alperin
en el libro \cite{gems}, p. 113--114.   Recomiendo vivamente este texto a  todos los profesores universitarios
de Algebra Lineal: en el encontrar\'{a}n verdaderas joyas.

\endDocument 